\documentclass[a4paper, reqno]{amsart}
\usepackage[english]{babel}
\usepackage[T1]{fontenc}
\usepackage{mathptmx}
\usepackage{hyperref}
\usepackage{graphicx}
\hypersetup{unicode=true, pdfauthor=Jean-Paul Calvi, pdfstartview=FitH}
\newtheorem{theorem}{Theorem}[section]
\newtheorem{corollary}[theorem]{Corollary}
\newtheorem{lemma}[theorem]{Lemma}
\newtheorem{definition}[theorem]{Definition}

 \theoremstyle{remark}

\numberwithin{equation}{section}
\newcommand{\ZZ}{\mathbb{Z}} 
\newcommand{\NN}{\mathbb{N}} 
\newcommand{\RR}{\mathbb{R}} 

\begin{document}

\title{Lagrange interpolation at real projections of Leja sequences for the unit disk}
\author{Jean-Paul Calvi}
\address{Institut de Math\'{e}matiques de Toulouse, Universit\'{e} Paul Sabatier, Toulouse, France}
\email{jean-paul.calvi@math.ups-tlse.fr}
\author{Phung Van Manh}
\address{Institut de Mathématiques,
Université de Toulouse III and CNRS (UMR 5219), 31062, Toulouse Cedex 9, France and Department of Mathematics, Hanoi University of Education,
136 Xuan Thuy street, Caugiay, Hanoi, Vietnam}\email{manhlth@gmail.com}
\begin{abstract} We show that the Lebesgue constant of the real projection of Leja sequences for the unit disk grows like a polynomial. The main application is the first construction of explicit multivariate interpolation points in $[-1,1]^N$ whose Lebesgue constant also grows like a polynomial.
\end{abstract}
\maketitle

\section{Introduction}\label{sec:introduction} 
We pursue the work initiated in \cite{jpcpvm} that aims to construct \emph{explicit} and (or) \emph{easily computable} sets of efficient points for multivariate Lagrange interpolation by using the process of intertwining (see below) certain univariate sequences of points. Here, the efficiency of the interpolation points is measured by the growth of their Lebesgue constant (the norms of the interpolation operator). Namely, we look for sets of points $\mathbf{P}_n\subset \mathbb{R}^N$ --- for interpolation by polynomials of total degree at most $n$ --- for which the Lebesgue constants $\Delta(\mathbf{P}_n)$ grows at most like a polynomial in $n$. We say that such points are good interpolation points in the sense that if $\Delta(\mathbf{P}_n)=O(n^\alpha)$ with $\alpha\in \mathbb{N}^\star$ then, in view of a classical result of Jackson, the Lagrange interpolation polynomials at $\mathbf{P}_n$ of any function with $\alpha+1$ continuous (total) derivatives converge uniformly. This is detailed in the paper. In our former work \cite{jpcpvm}, multivariate interpolation points with good Lebesgue constants were constructed on the Cartesian product of many plane compact subsets bounded by sufficiently regular Jordan curves (including, of course, polydiscs) starting from Leja points for the unit disc (see below). Yet, from a practical point of view, especially if we have in mind applications to numerical analysis, the real case is more interesting. It is the purpose of this paper to exhibit explicit interpolation points in $[-1,1]^N$ with a Lebesgue constant growing at most like a polynomial. As far as we know, this is the first general construction of such points. This will be done by suitably modifying the methods employed in \cite{jpcpvm}. Actually, the unidimensional points will be taken as the projections on the real axis of the points of a Leja sequence. We shall first show how to describe (and compute) these points and, in particular, prove that they are Chebyshev-Lobatto points (of increasing degree) arranged in a certain manner. We shall then study their Lebesgue constant to prove that it grows at most like $n^3\log n$ where $n$ is the degree of interpolation. The passage to the multivariate case is identical to that shown in \cite{jpcpvm} and will not be detailed.

\smallskip

\textit{Notation.}  We refer to \cite{jpcpvm} for basic definitions on Lagrange interpolation theory. Let us just indicate that, given a finite set $A$, we write $w_A:=\prod_{a\in A} (\cdot -a)$. The fundamental Lagrange interpolation polynomial (FLIP) for $a\in A$ is denoted by $\ell(A,a;\cdot)$. We have $$\ell(A,a;\cdot)=\prod_{b\in A, b\neq a}(\cdot -b)/(a-b)= w_A/\big(w_A'(a)(\cdot-a)\big).$$ The Lagrange interpolation polynomial of $f$ is $\mathbf{L}[A;f]=\sum_{a\in A} f(a)\ell(A,a;\cdot)$ and the norm of $\mathbf{L}[A; \cdot]$ as an operator on $C(K)$ (where $K$ is a compact subset containing $A$) is the Lebesgue constant $\Delta(A)=\Delta(A,K)$. It is known that $\Delta(A)=\max_{z\in K} \sum_{a\in A} |\ell(A,a;z)|$. 
\par We shall denote by $M$, $M'$, etc constants independent of the relevant parameters.  Occurrences of the same letter in different places do not necessarily refer to the same constant.

\section{Leja sequences and their projections on the real axis} 

\subsection{Leja sequences for the unit disk} We briefly recall the definition and the structure of a Leja sequence for the unit disk $D=\{|z|\leq 1\}\subset\mathbb{C}$. A $k$-tuple $E_k=(e_0,\dots,e_{k-1})\in D^k $ with $e_0=1$ is a \emph{$k$-Leja section} for $D$ if, for $j=1,\dots ,k-1$, the $(j+1)$-st entry $e_{j}$ maximizes the product of the distances to the $j$ previous points, that is \begin{equation*}\prod_{m=0}^{j-1}  \left|e_j-e_m\right| = \max_{z\in D} \prod_{m=0}^{j-1} \left|z-e_m\right|, \quad j=1,\dots,k-1. \end{equation*} The maximum principle implies that the points $e_i$ actually lie on the unit circle $\partial D$. A sequence $E=(e_k\,:\,k\in\mathbb{N})$ for which $E_k:=(e_0,\dots,e_{k-1})$ is a $k$-Leja section for every $k\in\mathbb{N}^{\star}$ is called a \emph{Leja sequence} (for $D$). Of course, the points of a Leja sequence are pairwise distinct. 
\par
The structure of a Leja sequence is studied in \cite{biacal} where one can find the following result.
\begin{theorem}[Bia{\l}as-Cie{\.z} and Calvi]\label{th:StructureLeja} A Leja sequence is characterized by the following two properties. 
\begin{enumerate}
	\item  The underlying set of a $2^n$-Leja section for $D$ is formed of the $2^n$-th roots of unity. 
	\item If $E_{2^{n+1}}$ is a $2^{n+1}$-Leja section then there exist a $2^n$-root $\rho$ of $-1$ and a $2^n$-Leja section $E^{(1)}_{2^n}$ such that $E_{2^{n+1}}=(E_{2^n}\, ,\, \rho E^{(1)}_{2^n})$.
\end{enumerate}
\end{theorem}
Repeated applications of the above rule show that if $d=2^{n_0}+2^{n_1}+\dots+2^{n_r}$ with $n_0>n_1>\dots > n_r\geq 0$ then 
\begin{align} E_{d}&=(E_{2^{n_0}}\, ,\, \rho_0 E^{(1)}_{d-2^{n_0}})
=(E_{2^{n_0}}\, ,\, \rho_0 E^{(1)}_{2^{n_1}}\, , \, \rho_1\rho_0 E^{(2)}_{d-2^{n_0}-2^{n_1}})
\\&=\dots =(E_{2^{n_0}}\, ,\, \rho_0 E^{(1)}_{2^{n_1}}\, , \, \rho_1\rho_0 E^{(2)}_{2^{n_2}}, \dots ,\, \rho_{r-1}\cdots\rho_1\rho_0 E^{(r)}_{2^{n_r}}), \label{eqn:define.rho}\end{align}
where each $E^{(j)}_{2^{n_j}}$ consists of a complete set of $2^{n_j}$-roots of unity, arranged in a certain order (actually, a $2^{n_j}$-Leja section), and $\rho_j$ is a $2^{n_j}$-th root of $-1$.

\subsection{Projections of Leja sequences} We are interested in polynomial interpolation at the projections on the real axis of the points of a Leja sequence. We eliminate repeated values and this somewhat complicates the description of the resulting sequence. We use $\Re(\cdot)$ to denote the real part of a complex number (or sequence). 

\begin{definition} A sequence $X$ (in $[-1,1]$) is said to be a $\Re$-Leja sequence if there exists a Leja sequence $E=(e_k\,:\, k\in \mathbb{N})$ such that $X$ is obtained by eliminating repeated points in $\Re(e_k\,:\, k\in \mathbb{N})$. We write $X=X(E)$. \end{definition}

In particular $X(E)$ is a subsequence of $\Re(e_k\,:\, k\in \mathbb{N})$. Since for every $n\in \NN$, the underlying set of a $2^{n+1}$-Leja section is a complete set of $2^{n+1}$-st roots of unity (Theorem \ref{th:StructureLeja}), the corresponding real parts form the set $\mathcal{L}_{2^n}$ of Chebyshev-Lobatto (or Gauss-Lobatto) points of degree $2^n$,
\begin{equation*} \mathcal{L}_{2^{n}}=\{\cos (j\pi/2^{n})\,:\, j=0, \dots, 2^{n}\}.  
\end{equation*} 
These points are the extremal points of the usual Chebyshev polynomial (of degree $2^n)$ and are sometimes referred to as the ``Chebyshev extremal points''.  
\par For future reference, we state this observation as a lemma. 
\begin{lemma}\label{th:RLareGL} Let $X$ be a $\Re$-Leja sequence. For every $n\in \NN$, the underlying set of $X_{2^n+1}=(x_0,\dots,x_{2^n})$ is the set of Chebyshev-Lobatto points $\mathcal{L}_{2^n}$. 
\end{lemma}
Theorem \ref{thm:structure-projection} below gives two descriptions of $\Re$-Leja sequences. The first one is particularly adapted to the computations of $\Re$-Leja sequences when one is given  Leja sequences. Examples of easily computable (and explicit) Leja sequences can be found in \cite[Lemma 2]{jpcpvm}. In Figure \ref{fig:RLejapoints} (I), we show the first $16$ points of a Leja sequence $E$ and the first $9$ points of the corresponding $\Re$-Leja sequence $X(E)$. (The Leja sequence we use is given by the rule $E_2=(1,-1)$ and $E_{2^{n+1}}=(E_{2^n}, \exp(i\pi/2^n)E_{2^n})$.)

\smallskip

The concatenation of tuples is denoted by $\wedge$, 
$$(x_1,\dots,x_m)\wedge(y_1,\dots,y_n):=(x_1,\dots,x_m, y_1,\dots,y_n).$$
For every sequence of complex numbers $S=(s_k\,:\, k\in \mathbb{N})$ we define $S(j:k):=(s_j, s_{j+1}, \dots, s_k)$. As before, $S_k:=S(0:k-1)$.

\begin{theorem}\label{thm:structure-projection}
 A sequence $X=(x_k\, :\, k\in \mathbb{N})$ is a $\Re$-Leja sequence if and only if there exists a Leja sequence $E=(e_k\, :\, k\in \mathbb{N})$ such that
\begin{equation}\label{eq:cobaRLeja} X=(1,-1)\; \wedge \; \bigwedge_{j=1}^{\infty} \Re\Big(E(2^j:2^j+2^{j-1}-1)\Big).\end{equation}
Equivalently, $x_k=\Re(e_{\phi(k)})$, $k\in\mathbb{N}$, with $\phi(0)=0$, $\phi(1)=1$ and
\begin{equation}\label{eq:closedRLeja} \phi(k)=\begin{cases} \frac{3k}{2} -1 & k=2^n \\ 2^{\lfloor \log_2(k)\rfloor}+k-1 & k\neq 2^n \end{cases}, \quad k\geq 2, \end{equation}
where $\lfloor \cdot\rfloor$ is used for the ordinary floor function. 
\end{theorem}
\begin{proof} Let $X=X(E)$ with $E=(e_s\,:\, s\in \NN)$. We prove that if $2^j\leq k <2^j+2^{j-1}$ then $\Re(e_k)$ does not appear in $\Re (E_k)$ and therefore provides a new point for $X$. To do that, since $e_k$ itself does not belong to $E_k$, it suffices to check that $\overline{e_k}$ is not a point of $E_k$, equivalently $e_k\neq \overline{e_s}$, $0\leq s \leq k-1$. If $s<2^j$ then $\overline{e_s}$ is a $2^j$-th root of unity whereas $e_k$ is not. On the other hand, if $2^j\leq s\leq k-1$ then, in view of Theorem \ref{th:StructureLeja}, $e_k=\rho a$ and $e_s=\rho b$ where $\rho$ is a $2^j$-th root of $-1$ and both $a$ and $b$ are $2^{j-1}$-st roots of unity. The relation $e_k=\overline{e_s}$ yields $\rho/\overline{\rho}=\overline{b}/a$. The argument of the first number is of the form $2(2l+1)\pi/2^j$ and the argument of the second one is $2t\pi/2^{j-1}$ (with $l, \, t\in \ZZ$). Equality is therefore impossible.  
\par
Now, in view of Lemma \ref{th:RLareGL}, from $E_{2^{j+1}}$ we obtain $2^{j}+1$ points for $X$, namely the points in $\mathcal{L}_{2^j}$ arranged in a certain way. Yet, the $2^{j}+1$ first points of $X$ are already given by $E_{2^j}$ ($2^{j-1}+1$ points) together with, according to the first part of this proof, the $2^{j-1}$ points $\Re(e_k)$, $2^j\leq k <2^j+2^{j-1}$. This implies that if $2^j+2^{j-1}\leq k <2^{j+1}$ then $\Re(e_k)$ is not a new point for $X$. This achieves the proof of \eqref{eq:cobaRLeja}.

\par To prove \eqref{eq:closedRLeja} we observe that, in view of \eqref{eq:cobaRLeja}, we have $$\Re(e_{2^k+i}) = x_{2^{k-1}+i+1},\quad  0\leq i \leq 2^{k-1}-1.$$
Hence $\phi(2^{k-1}+i+1)=2^k+i$ and the expression for $\phi$ easily follows. 
 \end{proof}

\begin{corollary}[to the proof]\label{cor:structure.proj} 
If $X=X(E)$ then $X(2^n+1:2^{n+1})=\Re(E(2^{n+1}:2^{n+1}+2^{n}-1))$. 
\end{corollary} 

Decompositions \eqref{eqn:define.rho} and  \eqref{eq:cobaRLeja} are fundamental to this work. In particular, the binary expansion of $k$ will be used in the study of the tuple $X(0:k)$.   
\par
Note that, of course, the decomposition would be different if we projected the Leja points on another segment, say on $[-e^{i\theta}, e^{i\theta}]$. The distribution of the projected points in general depends on arithmetic properties of $\theta$. We shall not discuss the general case in this paper.
\par
Finally, let us point out that our $\Re$-Leja sequences are not Leja sequences for the interval. It can be shown that they are pseudo Leja sequences (see \cite{biacal} for the definition of pseudo Leja sequences). There is no known expression for Leja points for the interval. For that matter, such an expression is very unlikely to exist.  

\begin{figure}[htbp]
	\centering
	\begin{tabular}{p{0.47\textwidth}p{0.47\textwidth}}
		\begin{center}\includegraphics[width=0.45\textwidth]{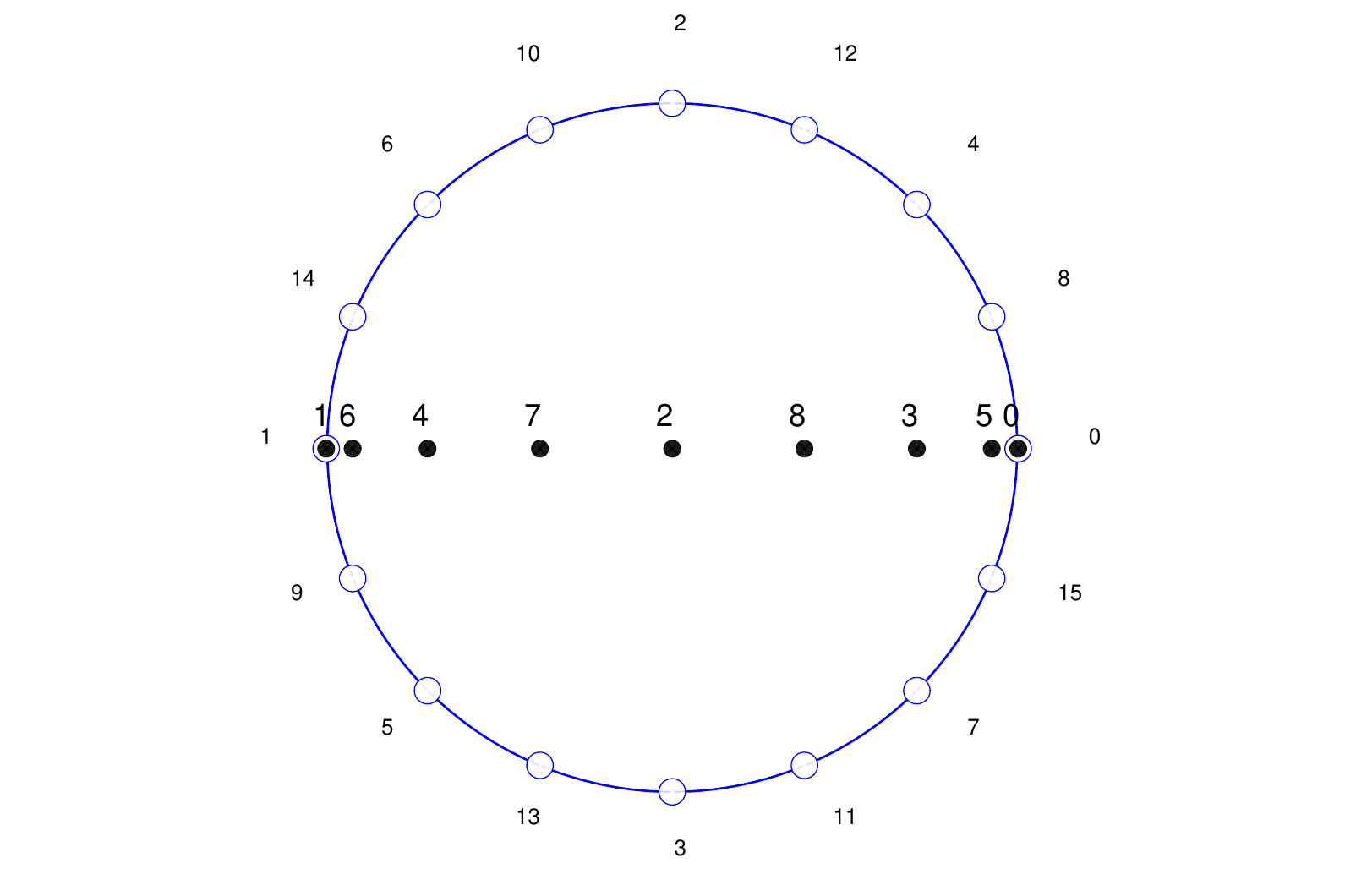}\end{center} & 	\begin{center}\includegraphics[width=0.45\textwidth]{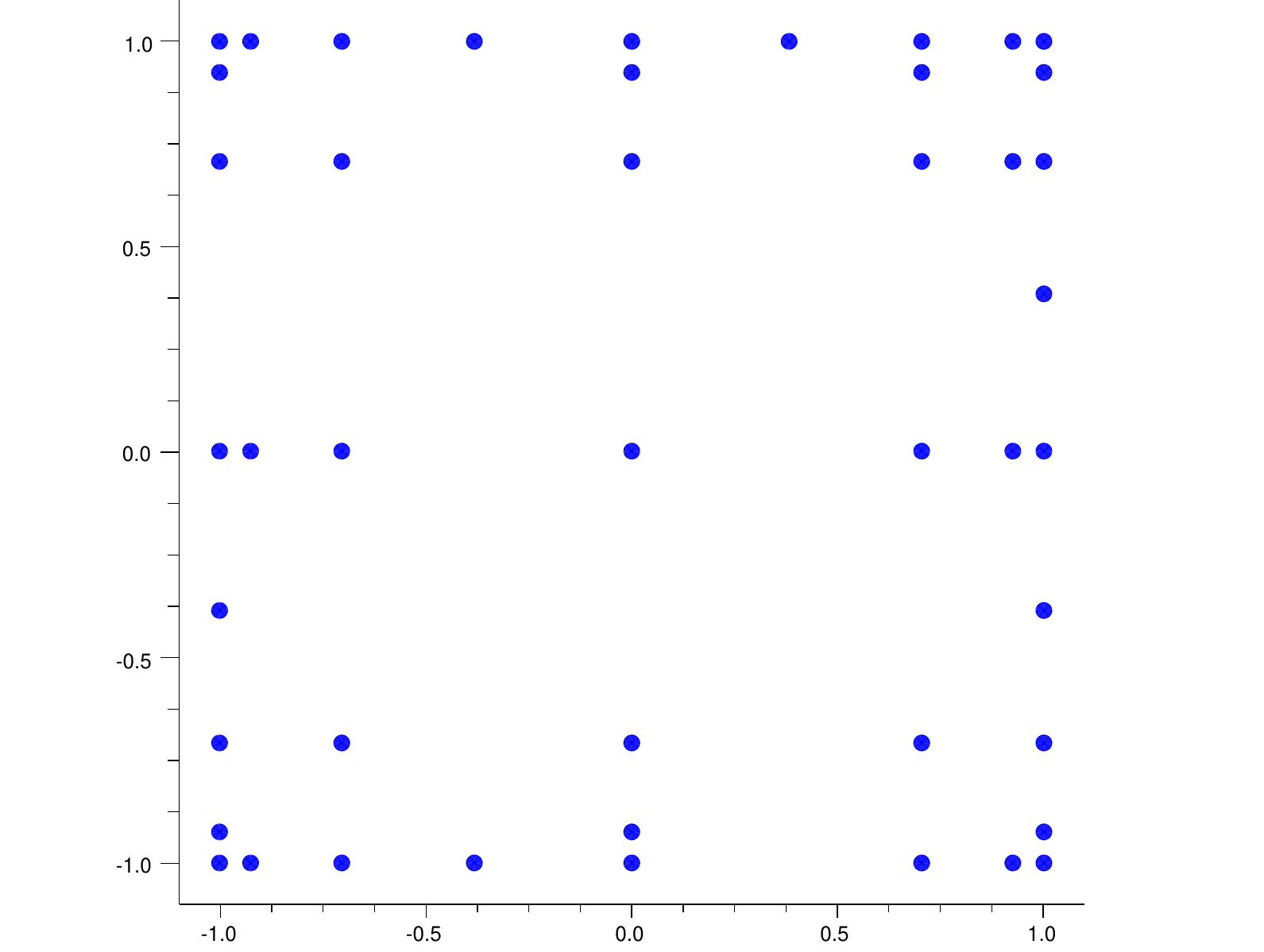}\end{center} \\
	\begin{footnotesize}	(I) First $9$ points of a $\Re$-Leja sequence. \end{footnotesize}& \begin{footnotesize}(II) $45$ interpolation points obtained as the intertwining of the points in (I) with themselves. \end{footnotesize}\\ 
		\end{tabular}
	\caption{Exemples of points from a $\Re$-Leja sequence and their intertwining.} \label{fig:RLejapoints}
\end{figure}
\section{Lebesgue constants of $\Re$-Leja sequences}
\subsection{The upper bound and its consequences} 
Recall that given a set $A$ of $n+1$ interpolation points in $[-1,1]$ and $f\in \mathrm{C}^s([-1,1])$, the Lebesgue inequality together with the Jackson theorem \cite[Theorem 1.5]{rivlin} yield the well known estimate 
\begin{equation*} \max_{[-1,1]} \|f-\mathbf{L}[A;f]\| = M (1+\Delta(A)) \; \omega(f^{(s)}, 2/n)/n^s\end{equation*}
where $\omega(f^{(s)}, \cdot)$ denotes the modulus of continuity of $f^{(s)}$ and $M$ does not depend either on $A$ or $n$.  The following theorem implies in particular that interpolation polynomials at the points of any $\Re$-Leja sequence converge uniformly on $[-1,1]$ to the interpolated function as soon as it belongs to $\mathrm{C}^4([-1,1])$. A weaker consequence is that the discrete measure $\mu_d:=\frac{1}{d+1}\sum_{i=0}^d [x_i]$ associated to the $\Re$-Leja sequence $X=(x_s\,:\, s\in \NN)$ weakly converges to the `$\arcsin$' distribution on $[-1,1]$ which is the equilibrium measure of the interval. Here $[x_i]$ denotes the Dirac measure at $x_i$. 

\begin{theorem}\label{thm:upper.bound}
 Let $X$ be a $\Re$-Leja sequence. The Lebesgue constant $\Delta(X_k)$ for the interpolation points $x_0, \dots, x_{k-1}$ satisfies the following estimate
\begin{equation*} \Delta(X_k)=O(k^3\log k), \quad k\to\infty. \end{equation*} \end{theorem}

The construction of good multivariate interpolation points is derived as follows. We start from $N$ $\Re$-Leja sequences  $X^{(j)}=(x^{(j)}_k: k\in \NN)$, $j=1,\dots,N$.  These $N$ sequences need not be distinct. We define $\mathbf{P}_k\subset [-1,1]^N$ as
\begin{equation*}
\mathbf{P}_k=\Big\{x_{\alpha}=(x^{(1)}_{\alpha_1},\ldots,x^{(N)}_{\alpha_N}): \sum_{j=1}^{N}\alpha_j\leq k \Big\}. 
\end{equation*}
It is known \cite{calvi} that this is a valid set for interpolation by $N$-variables polynomials of degree at most $k$. Actually, $\mathbf{P}_k$ is the underlying set of the \emph{intertwining} of the univariate tuples $X^{(j)}(0:k)$, $j=1,\ldots, N$.
We refer to \cite{jpcpvm} for details on this definition and to \cite{calvi} for a general discussion of the intertwining process. Let us just emphasize that, in order to provide good points, the method requires to use sequences of interpolation points (we add one point when we go from degree $k$ to $k+1$) rather than arrays (the classical case : all the points change when we change degree). The reason for this is explained in \cite{jpcpvm}. \par In Figure \ref{fig:RLejapoints} (II), we show the points of a set $\mathbf{P}_8$ constructed with the first $9$ points of the $\Re$-Leja sequence in (I). 

\begin{theorem}\label{th:lebestmult} The Lebesgue constant $\Delta(\mathbf{P}_k)$ grows 
at most like a polynomial in $k$ as $k\to\infty$. \end{theorem}  
\begin{proof} The proof of \cite[Theorem 16]{jpcpvm} works as well in this case.  \end{proof}
Just as in the univariate case, the multivariate versions of the Lebesgue inequality and the Jackson theorem together with Theorem \ref{th:lebestmult} imply that Lagrange interpolants of sufficiently smooth functions converge uniformly. Examining the terms in the proof of \cite[Theorem 16]{jpcpvm}, we find
\begin{equation*}
\Delta(\mathbf{P}_k)=O\Big(k^{(N^2+11N-6)/2}\log ^N k\Big), \quad k\to\infty,
\end{equation*}
which gives a more precise idea of the required level of smoothness. 
This bound however is certainly pessimistic. 
\subsection{Outline of the proof of Theorem \ref{thm:upper.bound}} We take advantage of the structure of the points of a $\Re$-Leja sequence. The first step is a simple algebraic observation. We use the notation recalled at the end of the introduction. 
\begin{lemma}\label{th:alg_rel} Let $N=N_0\cup \cdots \cup N_{s-1}$ where the $N_i$ form a partition of the finite set $N\subset \RR$. We have 
\begin{equation}\label{eq:relflip} \ell(N, a; \cdot)= \frac{w_{N\setminus N_i}}{w_{N\setminus N_i}(a)} \ell (N_i,a; \cdot), \quad a\in N_i,\quad i=0,\dots,s-1.\end{equation}
Consequently, 
\begin{equation}\label{eq:rellebconst} \Delta(N) \leq \sum_{i=0}^{s-1} \max_{x\in K,\, a\in N_i}\left|\frac{w_{N\setminus N_i}(x)}{w_{N\setminus N_i}(a)}\right|\; \Delta(N_i),\end{equation} 
where the Lebesgue constants are computed with respect to a compact set $K$ containing $N$.
\end{lemma}
\begin{proof} We readily check that the polynomial on the right hand side of \eqref{eq:relflip} satisfies the defining properties of $\ell(N,a;\cdot)$. The estimate for the Lebesgue constant $\Delta(N)$ follows from the definition together with the formula for the FLIPs. 
\end{proof}
Given a $\Re$-Leja sequence $X=X(E)$, to estimate $\Delta (X_k)$, $2^n+1<k\leq 2^{n+1}$, we shall first apply the lemma with a partition of $X_k$ into two subsets, namely $A=X(0:2^n)=\mathcal{L}_{2^n}$ and $B=X(2^n+1:k-1)$. Here and below, when there is no risk of misunderstanding, we confuse a tuple with its underlying set. In other words, when we write $T=Y$ with $T$ a tuple and $Y$ a set, we mean that $Y$ is formed of the entries of $T$. Our choice for $A$ and $B$, of course, is motivated by the fact that the Chebyshev-Lobatto points are excellent interpolation points for which there is a large amount of information available, see below. With this choice, Lemma \ref{th:alg_rel} gives
\begin{equation}\label{eq:rellebconst2}\Delta(X_k)\leq \Delta(A) \max_{x\in [-1,1],\, a\in A} \left|\frac{w_B(x)}{w_B(a)}\right|+ \Delta(B) \max_{x\in [-1,1],\, b\in B} \left|\frac{w_A(x)}{w_A(b)}\right|. \end{equation} 
The factors depending on the first subset $A$ in \eqref{eq:rellebconst2} will be easily estimated. The more difficult part will be to estimate $w_B$ and $\Delta_B$. To do that, we shall use a partition of $B$ and still have recourse to Lemma \ref{th:alg_rel} in its most general form. We point out however that our method is unlikely, it seems, to give the best estimates. Intuitively, we think that sharp estimates cannot be obtained by separating the interpolation points into two or more groups.\par It is not difficult to see that the Lebesgue constant $\Delta(X_k)$ cannot grow slower than $k$. This is explained below in subsection \ref{sec:lowerbound}.      
\subsection{Interpolation at Chebyshev-Lobatto points}
We collect a few results on Chebyshev-Lobatto points. First, since they are the extremal points of the Chebyshev polynomials, we have 
\begin{equation*} w_{\mathcal{L}_d}(x)=(x^2-1) d^{-1} T'_d(x), \quad x\in \mathbb{R}, \end{equation*}
where $T_d$ denotes the \emph{monic} Chebyshev polynomial of degree $d$. From $2^{d-1}T_d(\cos\theta)=\cos d\theta$ we readily find 
\begin{equation}\label{eq:wGLtrigform} w_{\mathcal{L}_d}(\cos \theta) = -2^{1-d} \sin \theta \sin d\theta,\quad \theta\in \mathbb{\RR}. \end{equation} 
A classical result of Ehlich and Zeller \cite{zeller, brutman} ensures that 
\begin{equation}\label{eq:lebCL}\Delta(\mathcal{L}_d)=O(\log d), \quad d\rightarrow\infty.  \end{equation}

\begin{lemma}\label{th:wCL} Let $X=X(E)$ be a $\Re$-Leja sequence. If $2^n+1<k\leq 2^{n+1}$, $A=X(0:2^n)$,  $B=X(2^n+1:k-1)$ and $K=[-1,1]$ then 
\begin{equation*} \max_{x\in K,\,b\in B}\frac{|w_A(x)|}{|w_A(b)|} \leq 1/\sin(\pi/2^{n+1}). \end{equation*} 
\end{lemma}
\begin{proof} If $x=\cos t$ and $b=\cos \beta$ then, since $A=\mathcal{L}_{2^n}$, in view of \eqref{eq:wGLtrigform}, we have 
\begin{equation*}|w_A(x)/w_A(b)| =  |\sin t \sin(2^n t)|/|\sin \beta \sin(2^n\beta)| \leq 1 / |\sin \beta \sin(2^n\beta)|. \end{equation*} 
Hence, it suffices to check that \begin{equation}\label{eq:ineqsin} |\sin \beta \sin(2^n\beta)|\geq \sin(\pi/2^{n+1}).\end{equation} Since $b=x_j$ with $2^n+1\leq j \leq k-1<2^{n+1}$, equation \eqref{eq:closedRLeja} gives
\begin{equation*} x_j= \Re(e_{2^n+j-1})=\Re(e_{2^{n+1}+u-1}), \quad u=j-2^n, \; 1\leq u\leq 2^n-1.  \end{equation*} Theorem \ref{th:StructureLeja} says that $e_{2^{n+1}+u-1}=\rho z$ where $\rho$ is a $2^{n+1}$-st root of $-1$ and $z$ is a $2^n$-th root of $1$. This means that the angle $\beta$ that gives $b=x_j$ is of the form
\begin{equation*} \beta = (2\tau+1)\pi/2^{n+1} + 2\tau'\pi/2^n \quad \text{with $\tau,\tau'\in \ZZ$}. \end{equation*}
It follows that $|\sin 2^n\beta|=1$ and $|\sin \beta|\geq \sin (\pi/2^{n+1})$. This gives inequality \eqref{eq:ineqsin} and  concludes the proof of the lemma.  \end{proof}
\subsection{A lower bound}\label{sec:lowerbound}
We now show that when we remove one point from $\mathcal{L}_d$, the Lebesgue constant grows significantly faster.  Write $a_i=\cos (i\pi/d)$. Suppose that $A_j:=\mathcal{L}_d\setminus\{a_j\}$ with $1\leq j\leq d-1$ (so that we remove a point different from $1$ and $-1$). We compute the values of the FLIPs for $A_j$ at the missing point $a_j$. From $w_{\mathcal{L}_d}(\cdot)=w_{A_j}(\cdot)(\cdot - a_j)$ we get
$$ w_{A_j}(a_j)=w'_{\mathcal{L}_d}(a_j)\quad\text{and}\quad w'_{\mathcal{L}_d}(a_i)= w'_{A_j}(a_i)(a_i-a_j), \quad i\neq j. $$
Hence,
$$\ell(A_j,a_i;a_j)=\frac{w_{A_j}(a_j)}{(a_j-a_i)w'_{A_j}(a_i)}=-\frac{w'_{\mathcal{L}_d}(a_j)}{w'_{\mathcal{L}_d}(a_i)}.$$ 
Yet, as easily follows from \eqref{eq:wGLtrigform}, $$w'_{\mathcal{L}_d}(a_i)=\pm 2^{1-d}d, \quad i=1,\dots,d-1.$$ Hence $|\ell(A_j,a_i;a_j)|=1$ for $d-2$ values of $i$, namely for $i=1,\dots,d-1$, $i\neq j$. Consequently,
\begin{equation}\label{eq:lebconsonemis}\Delta(A_j) \geq \sum_{i=1, i\neq j}^{d-1} |\ell(A_j,a_i;a_j)|= d-2.\end{equation}
Here is the consequence about our $\Re$-leja sequences. If $X$ is a $\Re$-Leja sequence, then $X(0:2^n-1)$ is formed of all the Chebyshev-Lobatto points of degree $2^n$ with only one missing and this missing point is different from $1$ and $-1$ (as soon as $n\geq 2$) which are the first two points of the sequence. Hence, according to \eqref{eq:lebconsonemis},  $\Delta(X(0:2^n-1))\geq 2^n-2$ which shows that the Lebesgue constant $\Delta(X_k)$ cannot grows slower than $k$.

\subsection{Interpolation at modified Chebyshev points}\label{sec:modchebpoints} We introduce other sets of interpolation points that will naturally come into play when dealing with $B=X(2^n+1:k-1)$. When $\cos\beta$ is not an extremal point of $T_d$, that is $\cos \beta\not\in \mathcal{L}_d$, then the equation $T_d(x)=T_d(\cos\beta)$ has $d$ roots in $[-1,1]$. The set of these roots --- which we call \emph{modified Chebyshev points} --- will by denoted by $\mathcal{T}_d^{(\beta)}$. Since the change of variable $x=\cos t$ transforms the equation in $\cos dt =\cos d\beta$, we have
\begin{equation*} \mathcal{T}_d^{(\beta)}=\{\cos \beta_j : \; \beta_j:= \beta + 2j\pi/d,\; j=0,\dots, d-1\}.\end{equation*} The following result is probably known but we are unable to provide references. 

\begin{lemma}\label{th:lebconsmodcheb} We have $\Delta(\mathcal{T}_d^{(\beta)})=O(\log d / |\sin d\beta|)$ as $d\to\infty$ and the constant involved in O does not depend on $\beta$. Equivalently, there exist $M$ such that, 
$$\Delta(\mathcal{T}_d^{(\beta)})\leq \log (d+1) / |\sin d\beta|), \quad d\geq 1.$$ 
\end{lemma} 
\begin{proof} First, since $$w_{\mathcal{T}_d^{(\beta)}}(x)= T_d(x)-T_d(\cos\beta),$$
see the introduction for the notation $w_{\mathcal{T}_d^{(\beta)}}$, we have 
\begin{equation*} \ell (\mathcal{T}_d^{(\beta)}, \cos \beta_j; x)= \frac{T_d(x)-T_d(\cos\beta)}{T'_d(\cos\beta_j) \; (x-\cos\beta_j)}=\frac{\sin \beta_j (\cos dt -\cos d\beta)}{d\sin d\beta (\cos t -\cos \beta_j)}, \quad x=\cos t.\end{equation*}
A use of the sum-to-product identity for cosines now yields
\begin{equation}\label{eq:trigstuff}\ell (\mathcal{T}_d^{(\beta)}, \cos \beta_j; x)= \frac{\sin \beta_j}{d\sin d\beta} \frac{\sin(d(t+\beta)/2)\sin(d(t-\beta)/2)}{\sin((t+\beta_j)/2)\sin((t-\beta_j)/2)}.\end{equation}
Yet, we also have $$\sin \beta_j = \sin ((t+\beta_j)/2)\cos((t-\beta_j)/2)-\sin ((t-\beta_j)/2)\cos((t+\beta_j)/2).$$ Using this in \eqref{eq:trigstuff}, we obtain after simplification, 
\begin{multline} \ell (\mathcal{T}_d^{(\beta)}, \cos \beta_j; x)= \frac{1}{d\sin d\beta}\left\{
\frac{\cos((t-\beta_j)/2)\sin(d(t+\beta)/2)\sin(d(t-\beta)/2)}{\sin((t-\beta_j)/2)}\right.
\\- \left.\frac{\cos((t+\beta_j)/2)\sin(d(t+\beta)/2)\sin(d(t-\beta)/2)}{\sin((t+\beta_j)/2)}\right\}
. \end{multline} 
It follows that
\begin{equation*} |\ell (\mathcal{T}_d^{(\beta)}, \cos \beta_j; x)|\leq \frac{1}{d|\sin d\beta|}\left\{ \left|\frac{\sin(d(t-\beta)/2)}{\sin((t-\beta_j)/2)}\right|+\left|\frac{\sin(d(t+\beta)/2)}{\sin((t+\beta_j)/2)}\right|\right\}, \quad x=\cos t.  
\end{equation*}
Hence, 
\begin{equation*} \Delta(\mathcal{T}_d^{(\beta)}) \leq \frac{1}{|\sin d\beta|} \left\{\max_{t\in \RR} F(t-\beta)+\max_{t\in \RR} F(t+\beta)\right\}=
 \frac{2}{|\sin d\beta|} \max_{t\in \RR} F(t),\end{equation*}
where $$F(t)=\frac{1}{d} \sum_{j=0}^{d-1} \left|\frac{\sin(dt/2)}{\sin((t-2j\pi/d)/2)}\right|.$$
But $\max_{t\in \RR} F(t)$ is exactly the Lebesgue constant for the $d$-th roots of unity which is known to be $O(\log d)$, see \cite{gronwall}.  
 \end{proof}
\section{Proof of Theorem \ref{thm:upper.bound}}
\subsection{Further reduction} We use Lemma \ref{th:wCL} and the classical estimate \eqref{eq:lebCL} of Ehlich and Zeller in \eqref{eq:rellebconst2} to obtain the following lemma.  

\begin{lemma} Let $X$ be a $\Re$-Leja sequence and let $2^n+1< k \leq 2^{n+1}$. If 
$A=X(0:2^n)$ and $B=X(2^n+1:k-1)$ then 
\begin{equation}\label{eq:pmtsecondstep} \Delta(X_k)\leq M \log2^n \max_{x\in [-1,1],\, a\in A}\frac{|w_B(x)|}{|w_B(a)|} + \frac{\Delta(B)}{\sin(\pi/2^{n+1})}, \end{equation}
where $M$ does not depend on $k$.  
\end{lemma}

The remaining required information is collected in the following two theorems. 

\begin{theorem}\label{th:tecstuf1} Let $X$ be a $\Re$-Leja sequence and let $2^n+1< k \leq 2^{n+1}$. If $A=X(0:2^n)$ and $B=X(2^n+1:k-1)$ then  
\begin{equation*}\max_{x\in [-1,1],\, a\in A}\frac{|w_B(x)|}{|w_B(a)|} \leq  2^{2n+2}.\end{equation*}
\end{theorem}
\begin{theorem}\label{th:tecstuf2} Let $X$ be a $\Re$-Leja sequence and let $2^n+1< k \leq 2^{n+1}$. If  $B=X(2^n+1:k-1)$ then 
\begin{equation*} \Delta(B) \leq M' 2^{2n}\log 2^n,\end{equation*}
where the constant $M'$ does not depend on $k$. 
\end{theorem}
\begin{proof}[End of proof of Theorem \ref{thm:upper.bound}] When $k-1$ is a power of $2$, the points of $X_k$ form a complete set of Chebyshev-Lobatto points and the bound is implied by Ehlich and Zeller's estimate \eqref{eq:lebCL}. We assume $2^n+1 < k \leq 2^{n+1}$. Using Theorems \ref{th:tecstuf1} and \ref{th:tecstuf2} in \eqref{eq:pmtsecondstep}, we obtain
\begin{equation}\label{eq:pmtthirdstep} \Delta(X_k) \leq M 2^{2n+2}\log 2^n  +  M' \frac{2^{2n}\log 2^n}{\sin(\pi/2^{n+1})}.\end{equation}
Since $n=\lfloor\log_2(k)\rfloor$ (or $\lfloor\log_2(k)\rfloor-1$ in the case $k=2^{n+1}$) and $1/\sin(\pi/2^{n+1})=O(2^n)$, this readily gives the existence of a constant $M''$ (independent of $k$) such that
\begin{equation*} \Delta(X_k) \leq M'' k^3\log k, \quad \text{for $k$ large enough} .\end{equation*} Observe that the highest power (that is, $k^3$) comes from the second term in \eqref{eq:pmtthirdstep}.
 \end{proof}
 \subsection{A trigonometric inequality}
 The proofs of the two remaining steps rest on an elementary inequality that we present in this subsection. As in \cite{jpcpvm}, the key observation is
  \begin{equation}\label{eq:ineqbasis}|\sin\alpha|\geq |\sin 2^n\alpha|/2^n, \quad n\in \NN,\quad \alpha\in \RR.\end{equation}
\begin{lemma}\label{lem:ineq+}
Let $r\geq 1$ and let $n_0>n_1>\cdots>n_r\geq 0$ be a finite decreasing sequence of natural numbers. If $2^{n_j}\varphi_j=\pi\; [2\pi]$ (i.e. $2^{n_j}\varphi_j=\pi$ mod $2\pi$), $j=0,\dots,r-1$,
then 
\begin{equation}\label{for:trigono-inequ}
\prod_{j=0}^{r-1}|\sin 2^{n_{j+1}-1}(\varphi-\varphi_0-\cdots-\varphi_j)|\geq (1/2^{n_0-n_r})|\cos 2^{n_0-1}\varphi|, \quad \varphi\in \RR.
\end{equation}
\end{lemma} 
\begin{proof} The proof is by induction. To treat the case $r=1$, we prove that 
$$|\sin 2^{n_1-1}(\varphi-\varphi_0)|\geq (1/2^{n_0-n_1})|\cos 2^{n_0-1}\varphi|.$$
 Using \eqref{eq:ineqbasis} with $\alpha=2^{n_1-1}(\varphi-\varphi_0)$ and $n=n_0-n_1$ we obtain 
\begin{equation*}
|\sin 2^{n_1-1}(\varphi-\varphi_0)|
\geq (1/2^{n_0-n_1})|\sin 2^{n_0-1}(\varphi-\varphi_0)|.  
\end{equation*}
But, since $2^{n_0}\varphi_0=\pi \; [2\pi]$, $|\sin 2^{n_0-1}(\varphi-\varphi_0)|=|\cos  2^{n_0-1}\varphi|$ and the claim follows.  \par
We now assume that the inequality is true for $r=k$ and prove it for $r=k+1$. The induction hypothesis applied to $\varphi-\varphi_0$ instead of $\varphi$ yields
\begin{equation}\label{for:trigono-inequ1}
\prod_{j=1}^{k}|\sin 2^{n_{j+1}-1}\big((\varphi-\varphi_0)-\varphi_1-\cdots-\varphi_j\big)|\geq (1/2^{n_1-n_{k+1}})|\cos 2^{n_1-1}(\varphi-\varphi_0)|, \quad \varphi\in \RR.
\end{equation} 
multiplying by the term corresponding to $j=0$, we obtain 
\begin{multline}\label{for:trigono-inequ2}
\prod_{j=0}^{k}|\sin 2^{n_{j+1}-1}\big((\varphi-\varphi_0)-\varphi_1-\cdots-\varphi_j\big)|\\ \geq \frac{1}{2^{n_1-n_{k+1}}}|\sin 2^{n_1-1}(\varphi-\varphi_0) \cos 2^{n_1-1}(\varphi-\varphi_0)|\\
 = \frac{1}{2^{n_1-n_{k+1}+1}}|\sin 2^{n_1}(\varphi-\varphi_0)|, \quad \varphi\in \RR.
\end{multline} 
Another use of \eqref{eq:ineqbasis} with $n=n_0-n_1-1$ shows that
\begin{equation*} \frac{1}{2^{n_1-n_{k+1}+1}}|\sin 2^{n_1}(\varphi-\varphi_0)| \geq \frac{1}{2^{n_0-n_{k+1}}} |\sin 2^{n_0-1}(\varphi-\varphi_0)|, \quad \varphi\in \RR. \end{equation*} The sine on the right hand side is shown to be $|\cos  2^{n_0-1}\varphi|$ as in the case $r=1$. 
\end{proof}
\subsection{Proof of theorem \ref{th:tecstuf1}}
Let $X=X(E)$ and $B=X(2^n+1:k-1)$ with $2^n+1 < k \leq 2^{n+1}$. We write 
\begin{equation}\label{eq:not1} k-1=2^{n}+2^{n_1}+\dots+2^{n_r} \quad\text{with}\;n-1\geq n_1>\dots>n_r\geq 0,\end{equation}
and, to simplify the notation, 
\begin{align} n_0&=n+1,\\
d_i&=2^{n_0}+\cdots + 2^{n_i}, \quad i=0,\dots,r. \end{align}
Then, in view of Corollary \ref{cor:structure.proj}, we have 
\begin{align*} B&=X(2^{n}+1: 2^{n}+(d_1-d_0))
\,\wedge \; \bigwedge_{i=1}^{r-1} X\big(2^n+(d_i-d_0)+1:2^{n}+(d_{i+1}-d_0)\big)\\
&=\bigwedge_{i=0}^{r-1} \Re\big(E(d_i: d_{i+1}-1)\big).\end{align*}
Now using the structure properties of a Leja sequence, see Theorem \ref{th:StructureLeja} and \eqref{eqn:define.rho}, we see that the points of $X(2^n+1,k-1)$ are certain modified Chebyshev points (see subsection \ref{sec:modchebpoints}). Indeed, for $i\in \{0, \dots,r-1\}$, 
\begin{align} E(d_i: d_{i+1}-1) &= \rho_{i}\cdots\rho_1\rho_0E^{(i+1)}_{2^{n_{i+1}}}, \quad \text{with $\rho_j^{2^{n_j}}=-1$},\\
\Re(E(d_i: d_{i+1}-1))&= \mathcal{T}^{(\beta_0+\cdots+\beta_{i})}_{2^{n_{i+1}}},\quad\text{with $\beta_j=\arg \rho_j=(2t_j+1)\pi/2^{n_j}$, $t_j\in\ZZ$}. \end{align} 
This implies the following relation for the polynomial $w_B$, 
\begin{equation*} w_B(x)= \prod_{i=0}^{r-1} \Big\{T_{2^{n_{i+1}}}(x) -T_{2^{n_{i+1}}}(\cos (\beta_0+\cdots+\beta_{i}))\Big\}. \end{equation*}
It follows that for $a=\cos\varphi\in A=X(0:2^n)$
\begin{align}
\max_{x\in [-1,1]}\frac{|w_B(x)|}{|w_B(a)|}&=\max_{t\in\RR} \prod_{i=0}^{r-1} \frac{|\cos (2^{n_{i+1}}t) -\cos(2^{n_{i+1}}(\beta_0+\cdots+\beta_{i}))|}{|\cos (2^{n_{i+1}}\varphi) -\cos(2^{n_{i+1}}(\beta_0+\cdots+\beta_{i}))|} \\
&\leq 2^r \prod_{i=0}^{r-1} 1/\big|\cos (2^{n_{i+1}}\varphi) -\cos(2^{n_{i+1}}(\beta_0+\cdots+\beta_{i}))\big|. \label{eq:pastocos}
\end{align}
Now, a use of the sum-to-product formula for cosines together with two applications of Lemma \ref{for:trigono-inequ} (first with $\varphi_i=\beta_i$, then with $\varphi_i=-\beta_i$) enable us to bound the denominator in \eqref{eq:pastocos} and arrive to 
\begin{equation*} \max_{x\in [-1,1]}\frac{|w_B(x)|}{|w_B(a)|}\leq \frac{2^{2(n_0-n_r)}}{\cos^2 (2^{n_0-1}\varphi)}.\end{equation*} 
It remains to recall that $n_0=n+1$ so that $2^{2(n_0-n_r)}\leq 2^{2n+2}$ and observe that, since $A=\mathcal{L}_{2^n}$, $2^{n_0-1}\varphi=2^n\varphi=0\; [\pi]$ so that $\cos^2 (2^{n_0-1}\varphi)=1$. This concludes the proof of Theorem \ref{th:tecstuf1}. 
\subsection{Proof of theorem \ref{th:tecstuf2}} We still use the fact that, for $X=X(E)$ and $2^n+1<k\leq 2^{n+1}$, we have 
$B=\wedge_{i=0}^{r-1} B_i$, where the underlying set of $B_i$ is $\mathcal{T}^{(\beta_0+\cdots+\beta_{i})}_{2^{n_{i+1}}}$ with $\beta_j=(2t_j+1)\pi/2^{n_j}$, $t_j\in\ZZ$.

Using first Lemma \ref{th:alg_rel} (with $N_i=B_i$) and then Lemma \ref{th:lebconsmodcheb} to bound $\Delta(B_j)$, we obtain
\begin{equation}\label{eq:profftechsruff2}\Delta(B) \leq M \sum_{j=0}^{r-1} \max_{x\in [-1,1],\, a\in B_j}\frac{|w_{B\setminus B_j}(x)|}{|w_{B\setminus B_j}(a)|}\; \frac{\log(2^{n_{j+1}}+1)}{\left|\sin (2^{n_{j+1}}(\beta_0+\cdots+\beta_{j}))\right|}.\end{equation}
Now, just as in \eqref{eq:pastocos} (we just remove one factor), for $a=\cos \theta_j\in B_j$, we have
\begin{align}\label{eq:pastocos2}
\max_{x\in [-1,1]}\frac{|w_{B\setminus B_j}(x)|}{|w_{B\setminus B_j}(a)|}&=\max_{t\in\RR} \prod_{i=0, i\neq j}^{r-1} \frac{|\cos (2^{n_{i+1}}t) -\cos(2^{n_{i+1}}(\beta_0+\cdots+\beta_{i}))|}{|\cos (2^{n_{i+1}}\theta_j) -\cos(2^{n_{i+1}}(\beta_0+\cdots+\beta_{i}))|} \\
&\leq 2^{r-1} \prod_{i=0, i\neq j}^{r-1} 1/|\cos (2^{n_{i+1}}\theta_j) -\cos(2^{n_{i+1}}(\beta_0+\cdots+\beta_{i}))|. 
\end{align}
Again the sum-to-product formula for cosines transforms the right-hand side in a product of sines and it follows that the $j$-term in the right-hand side of \eqref{eq:profftechsruff2} is bounded by the maximum when $a=\cos \theta_j$ runs over $B_j$ of
\begin{equation}\label{eq:profftechsruff22} \log(2^{n_{j+1}}+1) \prod_{i=0, i\neq j}^{r-1} \left|\sin^{-1} 2^{n_{i+1}-1}(\theta_j-\beta_0-\cdots-\beta_i)\right| \prod_{i=0}^{r-1} \left|\sin^{-1} 2^{n_{i+1}-1}(\theta_j+\beta_0+\cdots+\beta_i)\right|.\end{equation}
Here we used the fact that
\begin{equation}\label{eq:groupsines}|\sin 2^{n_{j+1}-1}(\theta_j+\beta_0+\dots+\beta_j)|=|\sin 2^{n_{j+1}}(\beta_0+\dots+\beta_j)| \end{equation}
which enabled us to insert the isolated sine in \eqref{eq:profftechsruff2} into the second product of \eqref{eq:profftechsruff22}. To prove \eqref{eq:groupsines}, we observe that, since $a=\cos \theta_j\in B_j$, we have 
\begin{equation}\label{eq:newassump} 2^{n_{j+1}}\theta_j=2^{n_{j+1}}(\beta_0+\dots +\beta_j)\; [2\pi].\end{equation}
 
We now estimate independently both products in \eqref{eq:profftechsruff22}. The same bound is valid for every $a\in B_j$ and therefore provides an upper bound for the maximum over $B_j$ as required.  

\par\smallskip I) We start with the first product. In view of \eqref{eq:newassump}, since $n_{i+1}> n_{j+1}$ whenever $i<j$, we have
$$2^{n_{i+1}}(\theta_j-\beta_0-\cdots -\beta_i)=2^{n_{i+1}}(\beta_{i+1}+\cdots+\beta_j)\; [2\pi], \quad 0\leq i < j.$$
On the other hand, since $2^{n_s}\beta_s=\pi\; [2\pi]$, we also have 
$$2^{n_{i+1}}(\beta_{i+1}+\cdots+\beta_j)=\pi \; [2\pi], \quad 0\leq i < j.$$
Thus the absolute value of the first $j$ sines equals $1$ and we just need to estimate
\begin{equation}\label{eq:trigineqco2_2}
\prod_{i=j+1}^{r-1}|\sin 2^{n_{i+1}-1}(\theta_j-\beta_0-\cdots-\beta_i)|.
\end{equation}
To do that, we apply Lemma \ref{lem:ineq+} with $\varphi=\theta_j-\beta_0-\cdots-\beta_j$. We obtain the lower bound $(1/2^{n_{j+1}-n_r})|\cos 2^{n_{j+1}-1}(\theta_j-\beta_0-\cdots-\beta_j)|.$
Yet in view of \eqref{eq:newassump} this cosine equals $\pm1$ and we obtain 
\begin{equation}\label{eq:pts2_2}\prod_{i=0, i\neq j}^{r-1} \left|\sin^{-1} 2^{n_{i+1}-1}(\theta_j-\beta_0-\cdots-\beta_i)\right| \leq 2^{n_{j+1}-n_r}.
\end{equation}
Note that, in the case $j=r-1$, the whole product equals $1$ which obviously implies the inequality. The inequality is likewise satisfied in the case $r=1$ (for which the product is empty).  
\par\smallskip II) We now turn to the second product in \eqref{eq:profftechsruff22}. Using again Lemma \ref{lem:ineq+} with $\varphi=\theta_j+\beta_0$ as in \eqref{for:trigono-inequ2}, we get the upper bound 
$$2^{n_1-n_{r}+1}/|\sin 2^{n_1}(\theta_j+\beta_0)|.$$ 
However, since for every $s$, $\beta_s=(2t_s+1)\pi/2^{n_j}$ and $\theta_j=\beta_0+\dots+\beta_j+2q_j\pi/2^{n_{j+1}}$ with $t_s,\,q_j\in \ZZ$, we have
\begin{multline*}2^{n_1}(\theta_j+\beta_0)=2^{n_1}(2\beta_0+\beta_1+\cdots+\beta_j+2q_j\pi/2^{n_{j+1}})\\=2^{{n_1}+1}\beta_0+p\pi=\frac{(2t_0+1)\pi}{2^{n_0-n_1-1}}+p\pi \quad\text{with $p\in \ZZ$}. \end{multline*}
Note that since $n_0=n+1$ and $n>n_1$ we have $n_0-n_1-1>0$. 
This shows that $|\sin 2^{n_1}(\theta_j+\beta_0)|\geq \sin (\pi/ 2^{n_0-n_1-1})\geq 2/2^{n_0-n_1-1}=1/2^{n_0-n_1-2}$. We have therefore proved
\begin{equation}\label{eq:pts2_3} \prod_{i=0}^{r-1} \left|\sin^{-1} 2^{n_{i+1}-1}(\theta_j+\beta_0+\cdots+\beta_i)\right| \leq 2^{n_1-n_{r}+1} \cdot 2^{n_0-n_1-2}  =2^{n_0-n_r-1}\leq 2^n.\end{equation} 
\par\smallskip III) It remains to insert \eqref{eq:pts2_2} and \eqref{eq:pts2_3} in \eqref{eq:profftechsruff2} with the aid of \eqref{eq:profftechsruff22}. Indeed, we obtain  
\begin{align}\Delta(B) &\leq M \sum_{j=0}^{r-1} 2^{n_{j+1}-n_r} \cdot 2^n \cdot \log (2^{n_{j+1}}+1) \\
&\leq M 2^n \sum_{j=0}^{r-1} 2^{n_j+1}\log (2^{n_{j+1}}+1) \leq M 2^{n} \log (2^{n}+1) \sum_{j=0}^{r-1}  2^{n_{j+1}}=O(2^{2n}\log 2^n). \end{align}
This achieves the proof of Theorem \ref{th:tecstuf2}. 
 
\subsection*{Acknowledgement} The work of Phung Van Manh is supported by a PhD fellowship from the Vietnamese government. 
\bibliographystyle{plain}
\bibliography{bib_projection_leja}

\end{document}